\documentclass{article}

\usepackage{amssymb,amsmath,amsthm,amsfonts}
\usepackage[colorlinks=true,linkcolor=webgreen,filecolor=webbrown,citecolor=webgreen]{hyperref}
\usepackage[margin=1in]{geometry}
\usepackage[nobysame]{amsrefs}
\usepackage{xcolor}

\definecolor{webgreen}{rgb}{0,.5,0}
\definecolor{webbrown}{rgb}{.6,0,0}

\newcommand{\cG}{{\cal G}}
\newcommand{\cA}{{\cal A}}
\newcommand{\cB}{{\cal B}}
\newcommand{\cC}{{\cal C}}

\newcommand{\cQ}{{\cal Q}}
\newcommand{\ZZ}{\ensuremath{{\mathbb Z}}}
\DeclareMathOperator{\diam}{diam}
\newcommand{\MRev}[1]{~~\href{http://www.ams.org/mathscinet-getitem?mr=#1}{{\bf MR~#1}}}
\newcommand{\doi}[1]{\href{https://dx.doi.org/#1}{#1}}

\title{Modular Constructions of \(g\)-Golomb Rulers}
\author{Aditya Gupta}
\date{June 2026}

\begin{document}

\theoremstyle{plain}
\newtheorem{theorem}{Theorem}
\newtheorem{corollary}[theorem]{Corollary}
\newtheorem{lemma}[theorem]{Lemma}
\newtheorem{proposition}[theorem]{Proposition}
\theoremstyle{definition}
\newtheorem{definition}[theorem]{Definition}
\newtheorem{example}[theorem]{Example}
\theoremstyle{remark}
\newtheorem{remark}[theorem]{Remark}

\maketitle

\begin{abstract}
  A set \(\cG\) of integers is a \(g\)-Golomb ruler if each positive difference appears at most \(g\) times between any 2 elements of the set, and \(G(g,n)\) denotes the minimum diameter of such a ruler with \(n\) marks. We prove a general lemma for passing from certain modular constructions to ordinary \(g\)-Golomb rulers. Previous work used the average gap and the restriction on repeated gap lengths separately. We combine these two restrictions and solve the resulting extremal problem, giving a larger guaranteed cut than the average gap argument. We apply this lemma to cyclic relative difference sets, Singer sets, Ruzsa sets, and Paley quadratic residues to provide many competing constructions for \(g\)-Golomb rulers. A computation on the grid \(1\le g\le500\), \(n=g+b\), \(2\le b\le500\), compares the four resulting theorem bounds. A separate computation constructs explicit rulers at representative pairs, optimizes the multiplier, retained marks, and cut, and compares the resulting diameters with known values of \(G(g,n)\).
\end{abstract}

\section{Introduction}

A Golomb ruler is a finite set of integers whose positive pairwise differences are all distinct. Equivalently, every positive difference occurs at most once. These objects also appear as Sidon sets, \(B_2\)-sets, and Babcock sets, and are connected to additive number theory, combinatorial design theory, and frequency assignment problems \cite{1984.Atkinson&Hassenklover}, \cite{1986.Atkinson&Santoro&Urrutia}, and \cite{2004.Obryant}.

We first define the terminology used throughout the paper. The elements of a ruler are called its marks. For a nonempty finite set \(\cG\subset\ZZ\), its diameter is
\[
    \diam(\cG):=\max(\cG)-\min(\cG).
\]
For a finite set \(\cG\subset\ZZ\), write
\[
    r_{\cG-\cG}(d):=\#\{(a,c)\in \cG^2:c-a=d\}\qquad(d\ge1).
\]
The set \(\cG\) is a \(g\)-Golomb ruler if
\[
    r_{\cG-\cG}(d)\le g
\]
for every positive integer \(d\). A set of residues modulo \(N\) is a modular \(g\)-Golomb ruler if every nonzero residue modulo \(N\) occurs at most \(g\) times as a difference of two marks. Finally, \(G(g,n)\) denotes the minimum possible diameter of a \(g\)-Golomb ruler with \(n\) marks.

Many known upper bounds for \(G(g,n)\) come from modular constructions
\cite{2021.MartosDazaTrujillo}. A modular ruler gives marks on a circle, and cutting the circle at an empty gap turns it into an ordinary ruler. Martos et al.~\cite{2021.MartosDazaTrujillo} used the fact that each gap length occurs at most \(g\) times to obtain a lower bound on the number of distinct gaps. We combine this restriction with the fixed sum of the gaps and apply the sharper cut to several modular constructions.

The following are the main theorems of the paper.

\medskip
\noindent\textbf{Modular Golomb Rulers to \(g\)-Golomb Rulers}
Suppose there is a modular Golomb ruler with \(k\) marks modulo \(N\), and suppose \(g\mid N\). If
\[
    n\le k-\left\lfloor\frac{g-1}{2}\right\rfloor,
    \qquad
    n=ag+r,
    \qquad
    0\le r<g,
\]
then
\[
    G(g,n)
    \le
    \frac Ng-
    \left\lceil
        \frac{\frac Ng+\frac{g a(a-1)}2+ra}{n}
    \right\rceil.
\]
This is Theorem~\ref{thm:general-folding}. The term
\[
    \frac{g a(a-1)}2+ra
\]
is the improvement over the average gap cut.

\medskip
\noindent\textbf{Cyclic RDS construction:}
Let \(q\) be a prime power and let \(d\ge2\). Suppose
\[
    q^{d-2}\mid g,
    \qquad
    \eta:=\frac{g}{q^{d-2}}\mid q^d-1,
    \qquad
    \rho:=\gcd(\eta,q-1).
\]
Then the cyclic relative-difference-set construction gives, for every
\[
    n\le q^{d-1}-\frac{q^{d-2}(\eta-\rho)}2,
    \qquad
    n=ag+r,
    \qquad
    0\le r<g,
\]
the bound
\[
    G(g,n)
    \le
    \frac{q^{d-2}(q^d-1)}{g}
    -
    \left\lceil
    \frac{
        \frac{q^{d-2}(q^d-1)}{g}
        +\frac{g a(a-1)}2+ra
    }{n}
    \right\rceil.
\]
For \(d=2\), this recovers the Bose--Chowla quotient and its folded variants. This is Theorem~\ref{thm:cyclic-rds}.

\medskip
\noindent\textbf{Singer and Ruzsa constructions:}
If \(q\) is a prime power, \(g\mid q^2+q+1\), and
\[
    n\le q+1-\left\lfloor\frac{g-1}{2}\right\rfloor,
    \qquad
    n=ag+r,
    \qquad
    0\le r<g,
\]
then
\[
    G(g,n)
    \le
    \frac{q^2+q+1}{g}
    -
    \left\lceil
        \frac{\frac{q^2+q+1}{g}+\frac{g a(a-1)}2+ra}{n}
    \right\rceil.
\]
If \(p\) is an odd prime, \(g\mid p(p-1)\), and
\[
    n\le p-1-\left\lfloor\frac{g-1}{2}\right\rfloor,
    \qquad
    n=ag+r,
    \qquad
    0\le r<g,
\]
then
\[
    G(g,n)
    \le
    \frac{p(p-1)}{g}
    -
    \left\lceil
        \frac{\frac{p(p-1)}{g}+\frac{g a(a-1)}2+ra}{n}
    \right\rceil.
\]
When \(g\mid p-1\), no marks collide in the Ruzsa quotient, so the same bound is valid for every \(n\le p-1\). In particular,
\[
    G(g,p-1)
    \le
    \frac{p(p-1)}{g}
    -
    \left\lceil
        \frac{3p-1-g}{2g}
    \right\rceil.
\]
These are developed in Theorems~\ref{thm:folded-singer} and~\ref{thm:folded-ruzsa} and Corollary~\ref{cor:ruzsa-diagonal}.

\medskip
\noindent\textbf{Paley quadratic residues:}
For an odd prime \(p\ge5\), the nonzero quadratic residues modulo \(p\) form a modular \(g\)-Golomb ruler with \((p-1)/2\) marks, where
\[
    g=\left\lfloor\frac{p-1}{4}\right\rfloor.
\]
we have, for every \(n\le(p-1)/2\) with \(n=ag+r\), \(0\le r<g\),
\[
    G(g,n)
    \le
    p-
    \left\lceil
        \frac{p+\frac{g a(a-1)}2+ra}{n}
    \right\rceil.
\]
The same bound is valid for every larger target multiplicity.

\medskip
 Section~\ref{sec:extraction} proves the general extraction theorem. Section~\ref{sec:rds} gives the cyclic RDS construction and derives its bound. Section~\ref{sec:singer-ruzsa} records the Singer and Ruzsa bounds. Section~\ref{sec:paley} records the classical Paley quadratic-residue input and derives the resulting bound. Section~\ref{sec:computation} compares the four theorem bounds on a \(500\times499\) grid, and Section~\ref{sec:explicit-comparison} compares explicitly generated rulers at representative pairs.

\section{Modular Rulers to $g$-Golomb Rulers}\label{sec:extraction}

We use residues modulo \(N\) as marks on a circle of length \(N\). A set \(\cA\) of residues modulo \(N\) is a modular \(g\)-Golomb ruler modulo \(N\) if every nonzero residue modulo \(N\) occurs at most \(g\) times as a difference
\[
    a-a'\pmod N,
    \qquad
    a,a'\in\cA,
    \quad
    a\ne a'.
\]
The case \(g=1\) is a modular Golomb ruler.

\begin{lemma}\label{lem:fold-delete}
Let \(\cA\) be a modular Golomb ruler with \(k\) marks modulo \(N\), and suppose \(g\mid N\). Reduce every mark of \(\cA\) modulo \(N/g\). After deleting at most \(\lfloor (g-1)/2\rfloor\) marks, the remaining residues form a modular \(g\)-Golomb ruler modulo \(N/g\). Consequently, for every
\[
    n\le k-\left\lfloor \frac{g-1}{2}\right\rfloor,
\]
there is a modular \(g\)-Golomb ruler with \(n\) marks modulo \(N/g\).
\end{lemma}

\begin{proof}
Two distinct marks become equal after reduction modulo \(N/g\) precisely when their difference is a nonzero multiple of \(N/g\) modulo \(N\). The possible nonzero multiples are
\[
    \frac Ng,2\frac Ng,\ldots,(g-1)\frac Ng.
\]
The multiples \(tN/g\) and \((g-t)N/g\) describe the same collision pair, but in opposite directions. When \(g\) is even, the only self-opposite multiple is
\[
    \frac{g}{2}\cdot\frac Ng=\frac N2.
\]
This residue cannot occur as a difference of two distinct marks of \(\cA\). Indeed, if
\[
    a-a'\equiv\frac N2\pmod N,
\]
then also
\[
    a'-a\equiv\frac N2\pmod N.
\]
The same nonzero residue would then have two ordered representations, contradicting the fact that \(\cA\) is a modular Golomb ruler. Therefore the possible collision types split into at most
\[
    \left\lfloor\frac{g-1}{2}\right\rfloor
\]
opposite pairs.

Because \(\cA\) is a modular Golomb ruler modulo \(N\), each nonzero residue modulo \(N\) occurs as a difference at most once. Hence each opposite pair can cause at most one collision. Deleting one endpoint from each colliding pair removes all collisions, and deletes at most \(\lfloor (g-1)/2\rfloor\) marks.

It remains to bound difference multiplicities after folding. Fix a nonzero residue \(d\) modulo \(N/g\). The residues modulo \(N\) that reduce to \(d\) modulo \(N/g\) are
\[
    d,
    d+\frac Ng,
    d+2\frac Ng,
    \ldots,
    d+(g-1)\frac Ng.
\]
There are \(g\) residues, and each occurs as a difference among marks of \(\cA\) at most once. Therefore the difference \(d\) modulo \(N/g\) occurs at most \(g\) times after folding. The remaining set is a modular \(g\)-Golomb ruler modulo \(N/g\).
\end{proof}

\begin{lemma}\label{lem:gap-cut}
Let \(\cC\) be a modular \(g\)-Golomb ruler with \(n\) marks modulo \(N\). Write
\[
    n=ag+r,
    \qquad
    0\le r<g.
\]
Then there is an ordinary \(g\)-Golomb ruler with \(n\) marks and diameter at most
\[
    N-
    \left\lceil
        \frac{N+\frac{g a(a-1)}2+ra}{n}
    \right\rceil.
\]
The same bound is valid for every larger target multiplicity.
\end{lemma}

\begin{proof}
Place the \(n\) marks of \(\cC\) around a circle of length \(N\), and let
\[
    \delta_1,\ldots,\delta_n
\]
be the gaps between consecutive marks. These are positive integers satisfying
\[
    \delta_1+\cdots+\delta_n=N.
\]
Each gap length is a nonzero modular difference between two marks. Since \(\cC\) is a modular \(g\)-Golomb ruler, no positive gap length occurs more than \(g\) times.

Let
\[
    \Delta:=\max_i\delta_i.
\]
Among \(n=ag+r\) positive integers no larger than \(\Delta\), with no value repeated more than \(g\) times, the largest possible total sum is obtained by taking
\[
    g\text{ copies of }\Delta,
    \quad
    g\text{ copies of }\Delta-1,
    \quad\ldots,\quad
    g\text{ copies of }\Delta-a+1,
\]
and then \(r\) copies of \(\Delta-a\). Thus
\[
    N
    \le
    g\Delta+g(\Delta-1)+\cdots+g(\Delta-a+1)+r(\Delta-a).
\]
The right side equals
\[
    n\Delta-\frac{g a(a-1)}{2}-ra.
\]
Therefore
\[
    \Delta
    \ge
    \left\lceil
        \frac{N+\frac{g a(a-1)}2+ra}{n}
    \right\rceil.
\]

Cut at a gap of length \(\Delta\), remove that empty gap, and unwrap the remaining elements into a line. The resulting set has diameter at most \(N-\Delta\). Every ordinary difference in the unwrapped set is also a nonzero difference modulo \(N\), so it occurs at most \(g\) times. Hence the resulting set is an ordinary \(g\)-Golomb ruler. If the target multiplicity is larger than \(g\), the same set is still a valid ruler, which proves the final statement.
\end{proof}

\begin{theorem}\label{thm:general-folding}
Suppose there is a modular Golomb ruler with \(k\) marks modulo \(N\), and suppose \(g\mid N\). If
\[
    n\le k-\left\lfloor \frac{g-1}{2}\right\rfloor,
    \qquad
    n=ag+r,
    \qquad
    0\le r<g,
\]
then
\[
    G(g,n)
    \le
    \frac Ng-
    \left\lceil
        \frac{\frac Ng+\frac{g a(a-1)}2+ra}{n}
    \right\rceil.
\]
The same bound is valid for every larger target multiplicity.
\end{theorem}

\begin{proof}
By Lemma~\ref{lem:fold-delete}, folding modulo \(N/g\) and deleting collisions gives a modular \(g\)-Golomb ruler with \(n\) marks modulo \(N/g\). Lemma~\ref{lem:gap-cut} then converts it into an ordinary \(g\)-Golomb ruler with the claimed diameter.
\end{proof}

\begin{remark}
The standard average-gap cut only guarantees a removable gap of length at least \(\lceil N/n\rceil\). Lemma~\ref{lem:gap-cut} replaces this by
\[
    \left\lceil
        \frac{N+\frac{g a(a-1)}2+ra}{n}
    \right\rceil,
    \qquad
    n=ag+r,
    \qquad
    0\le r<g.
\]
The improvement is sharp if one only uses the fact that no gap length appears more than \(g\) times.
\end{remark}

\section{The cyclic RDS construction}\label{sec:rds}

We use the following classical cyclic relative-difference-set construction
of Elliott and Butson \cite{1966.ElliottButson}. This family includes the
usual Bose--Chowla quotient as the two-dimensional case.

Fix a prime power \(q\) and an integer \(d\ge2\), and define
\[
    N:=q^d-1.
\]
For a set \(\cB\) of residues modulo \(N\), write
\[
    r_{\cB-\cB}(t)
    :=
    \#\{(b,b')\in \cB^2:\ b-b'\equiv t \pmod N\}.
\]
The construction gives a set \(\cB\) of residues modulo \(N\) with
\[
    |\cB|=q^{d-1},
\]
such that
\[
    r_{\cB-\cB}\left(j\frac{N}{q-1}\right)=0
    \qquad
    \text{for }1\le j\le q-2,
\]
and
\[
    r_{\cB-\cB}(t)=q^{d-2}
\]
for every other nonzero residue \(t\) modulo \(N\). Equivalently, the
nonzero multiples
\[
    \frac{N}{q-1},\ 2\frac{N}{q-1},\ \ldots,\ (q-2)\frac{N}{q-1}
\]
are forbidden differences, while every other nonzero residue occurs
exactly \(q^{d-2}\) times as an ordered difference.

In standard notation, this is a cyclic relative difference set with
parameters
\[
    \left(
    \frac{q^d-1}{q-1},\ q-1,\ q^{d-1},\ q^{d-2}
    \right).
\]
\begin{theorem}\label{thm:cyclic-rds}
Let \(q\) be a prime power and let \(d\ge2\). Suppose
\[
    q^{d-2}\mid g,
    \qquad
    \eta:=\frac{g}{q^{d-2}}\mid q^d-1.
\]
Define
\[
    \rho:=\gcd(\eta,q-1).
\]
Then there is a modular \(g\)-Golomb ruler with at least
\[
    q^{d-1}-\frac{q^{d-2}(\eta-\rho)}2
\]
marks modulo
\[
    \frac{q^{d-2}(q^d-1)}{g}.
\]
Consequently, for every
\[
    n\le q^{d-1}-\frac{q^{d-2}(\eta-\rho)}2,
\]
writing
\[
    n=ag+r,
    \qquad
    0\le r<g,
\]
one has
\[
    G(g,n)
    \le
    \frac{q^{d-2}(q^d-1)}{g}
    -
    \left\lceil
    \frac{
        \frac{q^{d-2}(q^d-1)}{g}
        +\frac{g a(a-1)}2+ra
    }{n}
    \right\rceil.
\]
\end{theorem}

\begin{proof}
Let
\[
    N:=q^d-1.
\]
Start with the cyclic RDS set \(\cB\) described above. We reduce its residues modulo
\[
    \frac{N}{\eta}
    =
    \frac{q^{d-2}(q^d-1)}{g}.
\]
Two original marks become equal after this reduction exactly when their difference is one of the nonzero residues
\[
    \frac{N}{\eta},\ 2\frac{N}{\eta},\ \ldots,\ (\eta-1)\frac{N}{\eta}.
\]
The forbidden residues in the cyclic RDS are the nonzero multiples of
\[
    \frac{N}{q-1}.
\]
Among the \(\eta\) residues
\[
    0,\frac{N}{\eta},2\frac{N}{\eta},\ldots,(\eta-1)\frac{N}{\eta},
\]
exactly \(\rho=\gcd(\eta,q-1)\) are also multiples of \(N/(q-1)\). Hence exactly \(\eta-\rho\) of the possible collision residues are not forbidden. Each of these residues occurs exactly \(q^{d-2}\) times as a difference of elements of \(\cB\).

Therefore the number of colliding ordered pairs is
\[
    q^{d-2}(\eta-\rho).
\]
Opposite differences describe the same unordered collision pair. The only possible self-opposite collision residue is \(N/2\), and when it exists it is forbidden. Thus the number of unordered collision pairs is
\[
    \frac{q^{d-2}(\eta-\rho)}2.
\]
Delete one mark from each colliding pair. The remaining set has at least
\[
    q^{d-1}-\frac{q^{d-2}(\eta-\rho)}2
\]
marks, and its residues are distinct modulo \(N/\eta\).

It remains to check the difference multiplicity after reduction. Fix a nonzero residue modulo \(N/\eta\). It comes from exactly \(\eta\) residues modulo \(N\). Each lift either lies in the forbidden subgroup, in which case it occurs zero times as a difference of elements of \(\cB\), or lies outside it, in which case it occurs exactly \(q^{d-2}\) times. Hence, before deleting marks, the projected difference occurs at most
\[
    \eta q^{d-2}=g
\]
times. Thus the remaining projected set is a modular \(g\)-Golomb ruler with the stated number of marks and the stated modulus.

The bound follows by applying Lemma~\ref{lem:gap-cut} to this modular \(g\)-Golomb ruler and then passing to any subset of \(n\) marks.
\end{proof}

\begin{remark}
For \(d=2\), Theorem~\ref{thm:cyclic-rds} gives the standard Bose--Chowla construction. Indeed, then \(q^{d-2}=1\), so \(\eta=g\), and the theorem applies whenever
\[
    g\mid q^2-1.
\]
It gives at least
\[
    q-\frac{g-\gcd(g,q-1)}2
\]
marks modulo \((q^2-1)/g\). This recovers the usual no-deletion Bose--Chowla quotient when \(g\mid q-1\), and it refines the standard folded Bose--Chowla deletion bound in the remaining cases.
\end{remark}

\section{Singer and Ruzsa constructions}\label{sec:singer-ruzsa}

We next record the two classical modular Golomb ruler inputs used in the comparison. Singer gives \(q+1\) marks modulo \(q^2+q+1\) \cite{1938.Singer}. The Ruzsa--Spence construction gives \(p-1\) marks modulo \(p(p-1)\) \cite{1993.Ruzsa}*{Theorem~4.4}. In both cases, Theorem~\ref{thm:general-folding} followed by Lemma~\ref{lem:gap-cut} gives the same refined gap extraction.

\begin{theorem}\label{thm:folded-singer}
Let \(q\) be a prime power and suppose \(g\mid q^2+q+1\). If
\[
    n\le q+1-\left\lfloor \frac{g-1}{2}\right\rfloor,
    \qquad
    n=ag+r,
    \qquad
    0\le r<g,
\]
then
\[
    G(g,n)
    \le
    \frac{q^2+q+1}{g}
    -
    \left\lceil
        \frac{\frac{q^2+q+1}{g}+\frac{g a(a-1)}2+ra}{n}
    \right\rceil.
\]
\end{theorem}

\begin{proof}
Apply Theorem~\ref{thm:general-folding} with \(N=q^2+q+1\) and with the folding divisor equal to \(g\), using the Singer modular Golomb ruler.
\end{proof}

\begin{theorem}\label{thm:folded-ruzsa}
Let \(p\) be an odd prime and suppose \(g\mid p(p-1)\). If
\[
    n\le p-1-\left\lfloor \frac{g-1}{2}\right\rfloor,
    \qquad
    n=ag+r,
    \qquad
    0\le r<g,
\]
then
\[
    G(g,n)
    \le
    \frac{p(p-1)}{g}
    -
    \left\lceil
        \frac{\frac{p(p-1)}{g}+\frac{g a(a-1)}2+ra}{n}
    \right\rceil.
\]
\end{theorem}

\begin{proof}
Similar to Theorem~\ref{thm:folded-singer}, apply
Theorem~\ref{thm:general-folding} with \(N=p(p-1)\) and with the
folding divisor equal to \(g\), using the Ruzsa modular Golomb ruler.
\end{proof}

\begin{corollary}\label{cor:ruzsa-diagonal}
Let \(p\) be an odd prime, and suppose \(g\mid p-1\). Then
\[
    G(g,p-1)
    \le
    \frac{p(p-1)}{g}
    -
    \left\lceil
        \frac{3p-1-g}{2g}
    \right\rceil.
\]
In particular, for fixed \(g\),
\[
    G(g,p-1)
    \le
    \frac{p(p-1)}{g}-\frac{3p}{2g}+O_g(1).
\]
\end{corollary}

\begin{proof}
By Lemma~4 of Martos et al.~\cite{2021.MartosDazaTrujillo}, the Ruzsa
quotient for \(g\mid p-1\) is a modular \(g\)-Golomb ruler with \(p-1\)
marks modulo \(p(p-1)/g\). Apply Lemma~\ref{lem:gap-cut}. Since
\(p-1=ga\), we have \(r=0\). The removable gap has length at least
\[
    \left\lceil
        \frac{\frac{p(p-1)}{g}+\frac{g a(a-1)}2}{p-1}
    \right\rceil
    =
    \left\lceil
        \frac{3p-1-g}{2g}
    \right\rceil.
\]
Subtracting this gap from \(p(p-1)/g\) proves the exact bound.
\end{proof}

\section{Paley quadratic residue construction}\label{sec:paley}

We next use the Paley construction for quadratic residues.  The
case \(p\equiv3\pmod4\) is the Paley difference-set case \cite{1933.Paley}. The case
\(p\equiv1\pmod4\) is the Paley partial difference set case \cite{1994.Ma}.

Let \(p\ge5\) be an odd prime and let
\[
    \cQ_p=\{x^2\pmod p:x\not\equiv0\pmod p\}
\]
be the set of nonzero quadratic residues modulo \(p\).  For a residue
\(t\) modulo \(p\), write
\[
    r_{\cQ_p-\cQ_p}(t)
    :=
    \#\{(a,b)\in \cQ_p^2:b-a\equiv t\pmod p\}.
\]
Paley difference-set and partial-difference-set results give,
for every nonzero residue \(t\),
\[
    r_{\cQ_p-\cQ_p}(t)=
    \begin{cases}
    \dfrac{p-3}{4},
        & p\equiv3\pmod4,\\[6pt]
    \dfrac{p-5}{4},
        & p\equiv1\pmod4\text{ and }t\in\cQ_p,\\[6pt]
    \dfrac{p-1}{4},
        & p\equiv1\pmod4\text{ and }t\notin\cQ_p.
    \end{cases}
\]
Thus
\[
    r_{\cQ_p-\cQ_p}(t)
    \le
    \left\lfloor\frac{p-1}{4}\right\rfloor
\]
for every nonzero residue \(t\) modulo \(p\).

\begin{theorem}\label{thm:paley-modular}
Let \(p\ge5\) be an odd prime and put
\[
    g:=\left\lfloor \frac{p-1}{4}\right\rfloor.
\]
Then the nonzero quadratic residues \(\cQ_p\) form a modular \(g\)-Golomb
ruler modulo \(p\), with \((p-1)/2\) marks. Consequently, if
\(n\le (p-1)/2\) and
\[
    n=ag+r,
    \qquad
    0\le r<g,
\]
then
\[
    G(g,n)
    \le
    p-
    \left\lceil
        \frac{p+\frac{g a(a-1)}2+ra}{n}
    \right\rceil.
\]
\end{theorem}

\begin{proof}
The set \(\cQ_p\) has \((p-1)/2\) elements. The Paley difference
multiplicities show that every nonzero residue modulo \(p\) occurs at most
\[
    \left\lfloor \frac{p-1}{4}\right\rfloor = g
\]
times as a difference of two elements of \(\cQ_p\). Thus \(\cQ_p\) is a modular
\(g\)-Golomb ruler modulo \(p\). Applying Lemma~\ref{lem:gap-cut} gives
\[
    G(g,n)
    \le
    p-
    \left\lceil
        \frac{p+\frac{g a(a-1)}2+ra}{n}
    \right\rceil.
\]
\end{proof}

\section{Computational Analysis of Theorem Bounds}\label{sec:computation}

We compare the theorem bounds of the four construction families on the grid
\[
    1\le g\le500,
    \qquad
    2\le b\le500,
    \qquad
    n=g+b.
\]
The cases \(b\le1\) are omitted because they are trivial. Thus the comparison contains
\[
    500\cdot499=249{,}500
\]
parameter pairs.

Each parameter choice in
Theorems~\ref{thm:cyclic-rds}, \ref{thm:folded-singer},
\ref{thm:folded-ruzsa}, and~\ref{thm:paley-modular}
determines
\[
    h,\qquad k,\qquad M,
\]
where \(h\) is the multiplicity of the constructed modular ruler, \(k\) is
the maximum number of available marks, and \(M\) is its modulus. The
candidate can be used at \((g,n)\) whenever
\[
    h\le g
    \qquad\text{and}\qquad
    n\le k.
\]
Writing
\[
    n=ah+r,
    \qquad
    0\le r<h,
\]
Lemma~\ref{lem:gap-cut} gives
\[
    G(g,n)
    \le
    M-
    \left\lceil
        \frac{M+\frac{ha(a-1)}2+ra}{n}
    \right\rceil.
\]
For each grid point, the program evaluates this expression for every
applicable parameter choice and records the smallest bound supplied by
each of the four construction families. In total, the program enumerated
\(67{,}145\) applicable theorem candidates. The code used for this
computation is available at \cite{Gupta.ModularConstructions}.

\subsection{Cutoff Point}

The four theorems involve infinitely many primes and prime powers, so, for
the sake of a finite comparison, we must prove that sufficiently large
moduli cannot affect the results.

For each \(n\), let \(p\) be the smallest prime satisfying
\[
    p-1\ge n,
\]
and define
\[
    M_{\mathrm R}(n):=p(p-1).
\]
The original Ruzsa construction has \(p-1\) marks modulo
\(M_{\mathrm R}(n)\) and multiplicity \(1\). It therefore applies to every
target multiplicity \(g\ge1\). By Lemma~\ref{lem:gap-cut}, it gives the
bound
\[
    B_{\mathrm R}(n)
    :=
    M_{\mathrm R}(n)
    -
    \left\lceil
        \frac{M_{\mathrm R}(n)+\frac{n(n-1)}2}{n}
    \right\rceil
\]
at every point of the grid.

Any candidate that is strictly worse than \(B_{\mathrm R}(n)\) cannot be optimal in the final comparison. 
Now, consider any candidate from the four construction families having multiplicity \(h\), modulus \(M\), and at least \(n\) marks. Write
\[
    n=ah+r,
    \qquad
    0\le r<h.
\]
Lemma~\ref{lem:gap-cut} gives the bound
\[
    M-
    \left\lceil
        \frac{M+\frac{ha(a-1)}2+ra}{n}
    \right\rceil.
\]
For fixed \(h\) and \(n\), this bound is nondecreasing in \(M\). Indeed,
increasing \(M\) by \(1\) increases the ceiling term by either \(0\) or
\(1\), so the entire expression increases by either \(1\) or \(0\).

We can therefore determine exactly the last modulus that could match or
improve the initial Ruzsa bound. Fix \(h\) and \(n\). A candidate of
multiplicity \(h\) and modulus \(M\) can affect the final comparison only
if
\[
    M-
    \left\lceil
        \frac{M+\frac{ha(a-1)}2+ra}{n}
    \right\rceil
    \le B_{\mathrm R}(n).
\]
Solving this inequality for \(M\) gives
\[
    M
    \le
    B_{\mathrm R}(n)
    +
    \left\lfloor
        \frac{
            B_{\mathrm R}(n)+\frac{ha(a-1)}2+ra+n-1
        }{n-1}
    \right\rfloor
\]
The program evaluates the right-hand side for every relevant pair
\[
    1\le h\le500,
    \qquad
    h+2\le n\le1000.
\]
Its largest value is
\[
    1{,}017{,}072,
\]
attained when
\[
    (h,n)=(1,997),(1,998),(1,999),(1,1000).
\]

For example, take \(h=1\) and \(n=997\). The smallest prime satisfying
\(p-1\ge997\) is \(p=1009\), so
\[
    M_{\mathrm R}(997)
    =
    1009\cdot1008
    =
    1{,}017{,}072.
\]
The corresponding Ruzsa bound is
\[
\begin{aligned}
    B_{\mathrm R}(997)
    &=
    1{,}017{,}072
    -
    \left\lceil
        \frac{
            1{,}017{,}072+\frac{997\cdot996}{2}
        }{997}
    \right\rceil\\
    &=
    1{,}015{,}553.
\end{aligned}
\]
Since
\[
    997=997\cdot1+0,
\]
an arbitrary candidate of multiplicity \(1\) can match or improve this
bound only if
\[
\begin{aligned}
    M
    &\le
    1{,}015{,}553
    +
    \left\lfloor
        \frac{
            1{,}015{,}553
            +\frac{997\cdot996}{2}
            +996
        }{996}
    \right\rfloor\\
    &=
    1{,}015{,}553+1{,}519\\
    &=
    1{,}017{,}072
\end{aligned}
\]
Thus every candidate with modulus at least \(1{,}017{,}073\) is strictly
worse than the Ruzsa benchmark in this case. The same maximum is obtained
for \(n=998,999,1000\), while every other relevant pair \((h,n)\) gives a
smaller value.

Consequently,
\(
    1{,}017{,}073
\)
is the first modulus that is strictly worse than the Ruzsa benchmark in
every relevant case. By monotonicity, every candidate with
\[
    M\ge1{,}017{,}073
\]
is also strictly worse than that benchmark. Since the eventual best bound
among the four constructions is no larger than \(B_{\mathrm R}(n)\), no
such candidate can improve or tie the eventual best bound. It is
therefore sufficient to enumerate candidates satisfying
\[
    M<1{,}017{,}073.
\]

To obtain a finite range for generating primes, let \(N\) denote the modulus before taking a quotient.
In the Ruzsa and Singer constructions, \(N=hM\), while in the cyclic RDS
construction \(N=\eta M\) with \(\eta\le h\). Since \(h\le500\) and
\(M<1{,}017{,}073\), every relevant candidate satisfies
\[
    N<500\cdot1{,}017{,}073
     =508{,}536{,}500.
\]
For Ruzsa, \(N=p(p-1)\), which gives \(p\le22{,}551\). For Singer and
cyclic RDS, \(N\ge q^2-1\), which gives \(q\le22{,}550\). The Paley
condition \(h=\lfloor(p-1)/4\rfloor\le500\) gives the smaller bound
\(p\le2004\). The program therefore generates primes and prime powers
only up to \(22{,}552\).
\subsection{Comparison Methodology}

A strict win occurs when one construction gives the smallest bound at a
grid point and no other construction matches it. If two or more
constructions give the same smallest bound, the point is recorded once
in the ties row and is not counted as a strict win for any construction.

At a strict win, the winning margin is
\[
    \text{best competing bound}-\text{winning bound},
\]
and the relative winning margin is
\[
    \frac{\text{best competing bound}-\text{winning bound}}
         {\text{best competing bound}}.
\]
The mean margins in the second table are taken only over the points at
which the corresponding construction is the strict winner.

A construction is within \(1\%\) of the best when its bound is within
\(1.01\) times the best bound among the four constructions. Similarly, it
is within \(5\%\) of the best when its bound is within \(1.05\) times the
best bound. These percentages are calculated over all \(249{,}500\) grid
points. When several constructions tie for the best bound, every tied
construction is counted as being within both \(1\%\) and \(5\%\) of the
best.

\subsection{Results}

The strict winner counts are as follows.
\[
\begin{array}{|l|r|r|}
\hline
\text{Construction} & \text{Strict wins} & \text{Share} \\
\hline
\text{Cyclic RDS construction} & 102{,}416 & 41.05\% \\
\text{Ruzsa--Spence}            &  91{,}094 & 36.51\% \\
\text{Paley quadratic residues} &  54{,}804 & 21.97\% \\
\text{Singer}                   &   1{,}004 &  0.40\% \\
\text{Ties}                     &      182  &  0.07\% \\
\hline
\end{array}
\]
The theorem bound furnished by the cyclic RDS construction is the most frequent strict winner, followed by the bounds furnished by Ruzsa--Spence and Paley quadratic residues.

The comparative performance statistics are as follows.
\[
\begin{array}{|l|r|r|r|r|}
\hline
\text{Construction}
  & \text{Mean margin}
  & \text{Mean relative margin}
  & \text{Within }1\%\text{ of best}
  & \text{Within }5\%\text{ of best} \\
\hline
\text{Cyclic RDS construction} & 169.65 & 9.26\% & 91.98\% & 99.38\% \\
\text{Ruzsa--Spence}           &   6.06 & 0.24\% & 56.85\% & 71.06\% \\
\text{Paley quadratic residues}&  12.55 & 1.12\% & 34.62\% & 37.06\% \\
\text{Singer}                  & 492.68 & 3.85\% &  0.66\% &  1.19\% \\
\hline
\end{array}
\]

The cyclic RDS theorem bound performs best overall on this grid. It is the
most frequent strict winner and lies within \(5\%\) of the best bound at
\(99.38\%\) of all grid points. Moreover, when it wins, its mean relative
margin of \(9.26\%\) is substantially larger than that of the
Ruzsa--Spence or Paley bounds.

The Ruzsa--Spence bound remains highly competitive and gives the
second-largest number of strict wins. However, its mean relative winning
margin is only \(0.24\%\), so its strict improvements over the competing
bounds are generally small. The Paley bound gives a substantial number of
strict wins, although its strong performance is concentrated in a smaller
portion of the grid. The Singer bound wins rarely in this range, but its
large margins at the points where it does win show that it remains a useful
complementary bound.

These conclusions concern the closed-form estimates proved in the preceding
sections. They do not by themselves compare the shortest ordinary rulers
that can be extracted from the underlying modular constructions, since the
deletion and gap estimates may contain different amounts of slack.

\section{Explicit comparison of constructed rulers}
\label{sec:explicit-comparison}

We compare ordinary rulers obtained directly from the four modular
construction families. For each pair $(g,n)$, we generate eligible
modular parents and minimize the resulting ordinary diameter over unit
multipliers, retained marks, and cuts. We illustrate the procedure at
$(3,5)$ and then give several further examples.

The computation was carried out by a Python program archived at
\cite{Gupta.ModularConstructions}. It considers prime powers $q\le127$ for
cyclic RDS and Singer and primes $p\le127$ for Ruzsa--Spence and Paley. In the
cyclic RDS family, every parameter satisfying
\[
    \eta\mid q^d-1,
    \qquad
    \eta q^{d-2}\le3
\]
is included. All permitted Ruzsa--Spence and Singer quotient parameters of
nominal multiplicity at most $3$ are also included. Finite-field arithmetic
is performed using the Python package \texttt{galois}.

After forming each quotient, the program keeps the modular marks
and measures their difference multiplicity. Among the candidates in
the stated parameter range, it retains parents having at least $n$ marks
and actual multiplicity at most $g$. Thus a parent whose nominal
multiplicity exceeds $g$ can still be used if its measured multiplicity is
at most $g$. For each eligible parent, the program tests every unit
multiplier and every cyclic block of $n$ consecutive marks. Each such block determines a cut and an ordinary ruler. 

The selected cutoff of $127$ is determined by the largest displayed value $n=124$. $127$ is the smallest prime for which a Ruzsa parent has at least $124$ marks. The same range supplies cyclic RDS and Singer parents large enough for every displayed example. The computed minima are constrained to this finite parameter range, however larger parameters may yield further improvements.
\subsection{Example at \texorpdfstring{$(g,n)=(3,5)$}{(g,n)=(3,5)}}
\label{subsec:explicit-example-3-5}

We begin with
\[
    (g,n)=(3,5).
\]
For a modular ruler $\cA\subseteq\ZZ_M$ and a unit
$u\in\ZZ_M^{\times}$, write the elements of $u\cA$ in increasing order as
\[
    0\le b_0<b_1<\cdots<b_{k-1}<M
\]
and extend them periodically by $b_{i+k}=b_i+M$. For fixed $u$, the shortest
ordinary ruler obtained by retaining five marks has diameter
\[
    D_u(\cA)=\min_{0\le i<k}(b_{i+4}-b_i).
\]
It is enough to consider consecutive cyclic blocks. Indeed, if five selected
marks lie in a cyclic arc, that arc contains a block of five consecutive marks
of $u\cA$ whose span is no larger. We minimize $D_u(\cA)$ over all units $u$. The program performs the multiplier-and-window minimization for every eligible parent and selects the least resulting diameter within each family.

\paragraph{Cyclic RDS.}
Take $(q,d,\eta)=(7,2,3)$. For the primitive element $\alpha$, the trace-one exponents in $\ZZ_{48}$ are
\[
    \{j:\operatorname{Tr}_{\mathbb F_{49}/\mathbb F_7}(\alpha^j)=1\}
    =\{1,7,27,32,34,45,46\}.
\]
Reduction modulo $(7^2-1)/3=16$ gives
\[
    \cA_{\mathrm{RDS}}
    =\{0,1,2,7,11,13,14\}\subseteq\ZZ_{16}.
\]
Its measured multiplicity is $3$. Up to sign, the units modulo $16$ are
$1,3,5,7$, and their minimum five-mark spans are
\[
    5,\ 6,\ 6,\ 5,
\]
respectively. For $u=1$, the minimizing cyclic block is
$13,14,16,17,18$. Cutting before $13$ and translating gives
\[
    R_{\mathrm{RDS}}=\{0,1,3,4,5\},
    \qquad
    \diam(R_{\mathrm{RDS}})=5.
\]

\paragraph{Ruzsa--Spence.}
Take $p=7$, let $\theta=3$ be a primitive root modulo $7$, and take
$\tau=3$. The CRT construction gives
\[
    \{2,4,5,27,31,36\}\subseteq\ZZ_{42},
\]
where the mark indexed by $t$ is the simultaneous solution of
\[
    x_t\equiv t\pmod 6,
    \qquad x_t\equiv 3^t\pmod 7.
\]
Folding modulo $42/3=14$ gives
\[
    \cA_{\mathrm{R}}=\{2,3,4,5,8,13\}\subseteq\ZZ_{14},
\]
with measured multiplicity $3$. Using $u=1$ and retaining $2,3,4,5,8$ gives:
\[
    R_{\mathrm{R}}=\{0,1,2,3,6\},
    \qquad
    \diam(R_{\mathrm{R}})=6.
\]

\paragraph{Singer.}
Take $(q,\tau)=(7,3)$. The trace-zero construction in
$\mathbb F_{7^3}^{\times}/\mathbb F_7^{\times}\cong\ZZ_{57}$ gives the
eight exponents
\[
    \{8,15,19,38,48,50,51,56\}.
\]
After reduction modulo $57/3=19$, the exponents $19$ and $38$ coincide.
Retaining one copy gives
\[
    \cA_{\mathrm{S}}=\{0,8,10,12,13,15,18\}
    \subseteq\ZZ_{19},
\]
whose measured multiplicity is $3$. For $u=1$, the seven five-mark window
spans are
\[
    13,7,8,7,14,14,13.
\]
The block $8,10,12,13,15$ is optimal over all unit multipliers. Translating
by $-8$ gives
\[
    R_{\mathrm{S}}=\{0,2,4,5,7\},
    \qquad
    \diam(R_{\mathrm{S}})=7.
\]

\paragraph{Paley.}
For $p=13$, the nonzero quadratic residues form
\[
    \cA_{\mathrm{P}}=\{1,3,4,9,10,12\}\subseteq\ZZ_{13},
\]
with measured multiplicity $3$. Multiplication by $u=2$ gives
\[
    2\cA_{\mathrm{P}}=\{2,5,6,7,8,11\}.
\]
Retaining $2,5,6,7,8$ produces
\[
    R_{\mathrm{P}}=\{0,3,4,5,6\},
    \qquad
    \diam(R_{\mathrm{P}})=6.
\]

The diameter obtained from each construction family, together with the corresponding parameter choice, is summarized below.
\[
\begin{array}{|c|c|c|}
\hline
\text{Construction} & \text{Parameters} & \text{Diameter}\\
\hline
\text{Cyclic RDS} & (q,d,\eta)=(7,2,3) & 5\\
\text{Ruzsa--Spence} & (p,\tau)=(7,3) & 6\\
\text{Singer} & (q,\tau)=(7,3) & 7\\
\text{Paley} & p=13 & 6\\
\hline
\end{array}
\]

The exact value $G(3,5)=5$ is recorded in
\cite{2026.GuptaObryant}. Thus the cyclic RDS construction attains the optimum
at this pair.

\subsection{Further representative examples}
\label{subsec:representative-explicit}

We apply the same procedure to several additional examples. The source code is archived at
\cite{Gupta.ModularConstructions}.

A comprehensive table of known exact values of $G(g,n)$ is given by Gupta and
O'Bryant~\cite{2026.GuptaObryant}. The current sequences for multiplicities
$2$ and $3$ are also maintained as OEIS A392461 and A392462
\cite{OEIS.A392461} and \cite{OEIS.A392462}, respectively.

Let $D_{\mathrm{RDS}}$, $D_{\mathrm{R}}$, $D_{\mathrm{S}}$, and
$D_{\mathrm{P}}$ denote the least diameters found in the cyclic RDS, Ruzsa--Spence, Singer, and Paley families,
respectively.
\[
\begin{array}{|c|r|r|r|r|r|r|}
\hline
g & n & G(g,n) & D_{\mathrm{RDS}} & D_{\mathrm{R}} &
D_{\mathrm{S}} & D_{\mathrm{P}}\\
\hline
2 &   5 &   6 &    7 &    8 &   11 & 7\\
2 &  10 &  29 &   35 &   34 &   55 & \text{--}\\
2 &  18 & 111 &  122 &  132 &  216 & \text{--}\\
2 &  50 & \text{--} & 1081 & 1137 & 2094 & \text{--}\\
2 & 124 & \text{--} & 7036 & 7262 & 13991 & \text{--}\\
\hline
3 &   5 &   5 &    5 &    6 &    7 & 6\\
3 &  10 &  20 &   24 &   25 &   25 & \text{--}\\
3 &  18 &  73 &   81 &   96 &   85 & \text{--}\\
3 &  50 & \text{--} & 742 & 846 & 820 & \text{--}\\
3 & 124 & \text{--} & 4749 & 4853 & 4717 & \text{--}\\
\hline
\end{array}
\]

\subsection{Paley Construction Example}
\label{subsec:paley-explicit-example}

Paley has few eligible entries when $g\in\{2,3\}$ because its multiplicity is
$\lfloor(p-1)/4\rfloor$. To display its effectiveness consider instead
\[
    (g,n)=(4,9).
\]
For this example the same cutoff $127$ is used, with candidates of nominal
multiplicity at most $4$. The resulting diameters are
\[
\begin{array}{|c|c|c|c|c|c|}
\hline
(g,n) & G(g,n) & D_{\mathrm{RDS}} & D_{\mathrm{R}} &
D_{\mathrm{S}} & D_{\mathrm{P}}\\
\hline
(4,9) & 13 & 15 & 18 & 22 & 14\\
\hline
\end{array}
\]
where $G(4,9)=13$ is known from \cite{2026.GuptaObryant}.

The Paley entry comes from the nonzero quadratic residues modulo $19$:
\[
    \cQ_{19}=\{1,4,5,6,7,9,11,16,17\}\subseteq\ZZ_{19}.
\]
Since $19\equiv3\pmod4$, every nonzero difference occurs exactly
$(19-3)/4=4$ times, so all nine marks may be retained. Their cyclic gaps are
\[
    3,1,1,1,2,2,5,1,3.
\]
Cutting at the gap of length $5$ between $11$ and $16$ and translating $16$
to zero gives
\[
    R_{\mathrm{P},19}
    =\{0,1,4,7,8,9,10,12,14\},
    \qquad
    \diam(R_{\mathrm{P},19})=14.
\]
For comparison, Lemma~\ref{lem:gap-cut} applied to this same modular parent
gives the bound $16$. Thus the optimized Paley ruler is 1 above the
known optimum and the strict winner among the four families at this pair.

\section*{Acknowledgments}

The author thanks Prof. Kevin O'Bryant for invaluable guidance. We also thank the reviewers for their feedback.


\begin{bibdiv}
\begin{biblist}
\bib{1984.Atkinson&Hassenklover}{article}{
  date = {1984},
  author = {Atkinson, M.~D.},
  author = {Hassenklover, A.},
  title = {Sets of integers with distinct differences},
  journal = {Sch. Comput. Sci.},
  address = {Carleton Univ., Ottawa, Ont., Canada},
  note = {Rep. SCS-TR-63},
}



\bib{1933.Paley}{article}{
  author = {Paley, R.~E.~A.~C.},
  title = {On orthogonal matrices},
  journal = {J. Math. Phys.},
  volume = {12},
  date = {1933},
  pages = {311--320},
  doi = {\doi{10.1002/sapm1933121311}},
}

\bib{1994.Ma}{article}{
  author = {Ma, Siu Lun},
  title = {A survey of partial difference sets},
  journal = {Des. Codes Cryptogr.},
  volume = {4},
  date = {1994},
  pages = {221--261},
  doi = {\doi{10.1007/BF01388454}},
}

\bib{1942.Bose}{article}{
  author = {Bose, R.~C.},
  title = {An affine analogue of Singer's theorem},
  journal = {J. Indian Math. Soc. (N.S.)},
  volume = {6},
  date = {1942},
  pages = {1--15},
}

\bib{1966.ElliottButson}{article}{
  author = {Elliott, J.~E.~H.},
  author = {Butson, A.~T.},
  title = {Relative difference sets},
  journal = {Illinois J. Math.},
  volume = {10},
  date = {1966},
  pages = {517--531},
}

\bib{1945.Bose&Chowla}{article}{
  author  = {Bose, R.~C.},
  author  = {Chowla, S.},
  title   = {On the construction of affine difference sets},
  journal = {Bull. Calcutta Math. Soc.},
  volume  = {37},
  date    = {1945},
  pages   = {107--112},
  review  = {\MRev{7,365g}},
}

\bib{2021.MartosDazaTrujillo}{article}{
    author = {Martos Ojeda, Carlos Andres},
    author = {Daza Urbano, David Fernando},
    author = {Trujillo Solarte, Carlos Alberto},
    title={Near-optimal $g$-Golomb rulers},
    journal={IEEE Access},
    year={2021},
    volume={9},
    pages={65482--65489},
    doi={\doi{10.1109/ACCESS.2021.3075877}}
}
\bib{1986.Atkinson&Santoro&Urrutia}{article}{
  date = {1986},
  author = {Atkinson, M.~D.},
  author = {Santoro, N.},
  author = {Urrutia, J.},
  title = {Integer sets with distinct sums and differences and carrier frequency assignments for nonlinear repeaters},
  journal = {IEEE Transactions on Communications},
  volume = {34},
  number = {6},
  pages = {614--617},
}
\bib{2004.Obryant}{article}{
  author={O'Bryant, Kevin},
  title={A complete annotated bibliography of work related to Sidon sequences},
  journal={Electron. J. Combin.},
  volume={DS11},
  date={2004},
  pages={39},
  review={\MRev{4336213}},
}
\bib{1993.Ruzsa}{article}{
  author = {Ruzsa, Imre Z.},
  title = {Solving a linear equation in a set of integers. {I}},
  journal = {Acta Arith.},
  volume = {65},
  number = {3},
  date = {1993},
  pages = {259--282},
  doi = {\doi{10.4064/aa-65-3-259-282}},
}
\bib{1938.Singer}{article}{
  author = {Singer, James},
  title = {A theorem in finite projective geometry and some applications to number theory},
  journal = {Trans. Amer. Math. Soc.},
  volume = {43},
  number = {3},
  date = {1938},
  pages = {377--385},
  doi = {\doi{10.1090/S0002-9947-1938-1501951-4}},
}
\bib{2026.GuptaObryant}{misc}{
  author = {Gupta, Aditya},
  author = {O'Bryant, Kevin},
  title = {Optimal diameters of high multiplicity \(g\)-Golomb rulers},
  date = {2026},
  note = {arXiv:2605.14229 [math.CO]},
}
\bib{Gupta.ModularConstructions}{misc}{
  author = {Gupta, Aditya},
  title = {Modular constructions for \(g\)-Golomb rulers},
  date = {2026},
  note = {GitHub repository, \url{https://github.com/AdityaGupta3011/Modular}},
}
\bib{OEIS.A392461}{misc}{
  author = {{OEIS Foundation Inc.}},
  title = {Entry {A392461} in {The On-Line Encyclopedia of Integer Sequences}},
  date = {2026},
  note = {\url{https://oeis.org/A392461}},
}
\bib{OEIS.A392462}{misc}{
  author = {{OEIS Foundation Inc.}},
  title = {Entry {A392462} in {The On-Line Encyclopedia of Integer Sequences}},
  date = {2026},
  note = {\url{https://oeis.org/A392462}},
}
\end{biblist}
\end{bibdiv}
\end{document}